\documentclass[11pt,reqno]{amsart}
\usepackage[T1]{fontenc}
\usepackage{lmodern}
\usepackage[margin=1in]{geometry}
\usepackage{amsmath,amssymb,mathtools,microtype}
\usepackage[hidelinks]{hyperref}
\hypersetup{pdftitle={The minimum number of subset sums of a zero-sum-free set},pdfauthor={Chao Liu}}
\usepackage{enumitem}
\usepackage{needspace}
\setlist[enumerate]{label=\textup{(\roman*)},leftmargin=2em}
\numberwithin{equation}{section}
\newtheorem{theorem}{Theorem}[section]
\newtheorem{lemma}[theorem]{Lemma}
\newtheorem{proposition}[theorem]{Proposition}
\newtheorem{corollary}[theorem]{Corollary}

\newtheorem{mainresult}{Theorem}
\theoremstyle{remark}

\newcommand{\Sig}{\Sigma_0}
\newcommand{\N}{\mathsf N}
\newcommand{\gen}[1]{\langle #1\rangle}
\newcommand{\stab}{\operatorname{Stab}}

\newcommand{\tri}[1]{T_{#1}}
\newcommand{\calH}{\mathcal H}
\DeclareMathOperator{\Ol}{Ol}
\allowdisplaybreaks[2]
\title[Subset sums of zero-sum-free sets]{The minimum number of subset sums\\ of a zero-sum-free set}
\author{Chao Liu}
\address{Department of Mathematics, University of Mississippi,
University, MS 38677, USA}
\email{chaoliuac@gmail.com}

\date{Revised November 27, 2026}

\begin{document}
\begin{abstract}
Let $f(k)$ be the minimum number of distinct nonempty subset sums of a zero-sum-free $k$-element set in a finite abelian group. We prove that
$f(k)=\lfloor k^2/2\rfloor+1$ for every $k\ge185$. Consequently, every finite abelian group $G$ of sufficiently large order satisfies $\Ol(G)\le\lceil\sqrt{2(|G|-1)}\rceil$, with equality attained, for instance, by every cyclic group of sufficiently large even order. For all sufficiently large $k$, we classify the sets attaining $f(k)$ in terms of consecutive cosets of a subgroup of order $2$, $4$, or $6$. We also determine the minimum at every feasible size in cyclic groups of order $p^\nu$ for all sufficiently large primes $p$, uniformly in $\nu$.
\end{abstract}
\maketitle

\section{Introduction}

Let $G$ be a finite abelian group, written additively. For a finite sequence
$T$ over $G$, let $\Sig(T)$ be the set of its subsequence sums, including
the empty sum. When $T$ is a set, these are its subset sums. We call $T$
\emph{zero-sum-free} if none of its nonempty subsequences has sum zero.
For such a sequence, put
\[
 \N(T)=|\Sig(T)|-1.
\]
Thus $\N(T)$ counts the distinct nonempty subsequence sums. We study
\[
 f(k)=\min\{\N(S): S\subseteq G,\ |S|=k,\ S\text{ zero-sum-free},\
                  G\text{ a finite abelian group}\}.
\]
The distinction between sets and sequences is essential in this definition:
the elements of $S$ must be distinct.

Eggleton and Erd\H{o}s introduced this invariant and proved, in particular,
that
\begin{equation}\label{eq:classical-upper}
 f(k)\le \left\lfloor\frac{k^2}{2}\right\rfloor+1
 \qquad(k\ge4);
\end{equation}
see \cite{EE}. Olson obtained a quadratic lower bound with coefficient
$1/9$ \cite{Olson}. Gao, Huang, Hui, Li, Liu, and Peng subsequently proved
the uniform estimate $f(k)\ge k^2/6$ \cite[Theorem 1.1]{GHHLLP}.
Put $T_j=j(j+1)/2$, $F(k)=\lfloor k^2/2\rfloor+1$, and $M(S)=|\Sig(S)|$. We prove that the Eggleton--Erd\H{o}s construction is optimal for all sufficiently large $k$.

\begin{mainresult}\label{th:exact}
There is an absolute $k_0$ such that $f(k)=F(k)$ for every $k\ge k_0$.
\end{mainresult}

The computer-assisted refinement summarized in Appendix~\ref{sec:explicit} gives $k_0=185$ as an admissible choice in Theorem~\ref{th:exact}. Thus
\[
 f(k)=\left\lfloor k^2/2\right\rfloor+1\qquad(k\ge185).
\]
The following classification has its own absolute threshold.

\begin{mainresult}\label{th:structure}
For all sufficiently large $k$, a zero-sum-free $k$-element set $S$ attains $f(k)$ if and only if its generated group $G$ has a subgroup $H$ with the following properties. Writing $h=|H|$, one of the following holds: $h=2$, or $h=4$ and $k\equiv2\pmod4$, or $h=6$ and $k\equiv3\pmod6$. The set $S\cap H$ has $h/2$ elements and $\Sig(S\cap H)=H$. The quotient $G/H$ is cyclic. With respect to some generator, write $k-h/2=hm+b$ with $0\le b<h$. Then $S\setminus H$ consists of all the lifts at heights $1,\ldots,m$ and $b$ distinct lifts at height $m+1$, and
\[
 |G/H|>hT_m+(m+1)b .
\]

\end{mainresult}

For a fixed group $G$, write $f_G(k)$ for the same minimum restricted to sets in $G$, with value $+\infty$ if no such set exists. We obtain a second exact formula.
\begin{mainresult}\label{th:pgroup}
There is an absolute $p_0$ such that for every prime $p\ge p_0$, every $\nu\ge1$, and every $k\ge1$,
\[
 f_{C_{p^\nu}}(k)=
 \begin{cases}
 T_k & \text{if } T_k<p^\nu,\\
 +\infty & \text{if } T_k\ge p^\nu .
 \end{cases}
\]
\end{mainresult}

We first prove a quantitative estimate valid at every size. In the proof of Theorem~\ref{th:exact} it excludes concentration in a proper subgroup.

\begin{theorem}\label{thm:main}
Let $S$ be a zero-sum-free set of $k\ge1$ elements in a finite abelian
group. Then
\begin{equation}\label{eq:main}
 \N(S)\ge \frac{k^2}{2}-13k\log(k+1)-116k,
\end{equation}
where $\log$ denotes the natural logarithm. Consequently,
\[
 \frac{k^2}{2}-O(k\log k)\le f(k)
 \le\left\lfloor\frac{k^2}{2}\right\rfloor+1.
\]
\end{theorem}

The associated zero-sum threshold is the \emph{Olson constant}
$\Ol(G)$: the smallest positive integer $t$ such that every $t$-element
subset of $G$ has a nonempty zero-sum subset.

\begin{corollary}\label{cor:olson}
For every finite abelian group of order $n$,
\[
 \Ol(G)\le\sqrt{2n}+13\log(n+1)+266.
\]
In particular, $\Ol(G)\le\sqrt{2n}+O(\log n)$, uniformly in $G$.
\end{corollary}

Hamidoune and Z\'emor proved
$\Ol(G)\le\sqrt{2|G|}+O(|G|^{1/3}\log|G|)$
\cite[Theorem 4.5]{HZ}. Thus the leading constant $\sqrt2$ in the Olson
problem is classical. Corollary~\ref{cor:olson} reduces the error in that
estimate to logarithmic order. Theorem~\ref{thm:main} controls the actual
number of subset sums; a bound on the ambient group order alone does not
give this conclusion.

The proof of Theorem~\ref{thm:main} has three parts. First, for every integer $r\ge2$, if $U$ is
a set and the sequence containing $2r$ copies of every element of $U$
is zero-sum-free, then $r$ copies already have at least
$r|U|(|U|+1)/2$ nonempty sums. Second,
we combine blocks of two and three copies to count sums in a quotient
whose multiplicities are almost equal. This yields an alternative for
antisymmetric zero-sum-free sets: either the set already has at least
half the square of its size in nonempty subset sums, or every proper
iterated sum of its symmetrization satisfies the expansion bound in
Proposition~\ref{prop:alternative}. Finally, a greedy argument converts
this expansion into \eqref{eq:main}. Its stopping rule handles tails
that generate a proper subgroup.

For the exact result we select two disjoint subsets whose sets of sums are nearly invariant under translation by each remaining element. An inverse theorem for iterated sumsets then yields consecutive cosets inside the sumset of the two selections. Extending this interval bounds the total height of the remaining elements. At equality, the capacity of the lowest cosets forces the configurations of Theorem~\ref{th:structure}.

OpenAI ChatGPT was used to assist in drafting and revising the manuscript.

\section{Notation and additive tools}

For a sequence $T=t_1\cdots t_\ell$, write
$\sigma(T)=t_1+\cdots+t_\ell$, with $\sigma(\varnothing)=0$.
For a subgroup $H\le G$, let $\pi_H:G\to G/H$ be the natural quotient map.
For an integer $r\ge1$ and a set $U$, the notation $U^r$ denotes the
sequence in which each element of $U$ occurs $r$ times. The sumset $rX$,
by contrast, denotes $X+\cdots+X$, with repetitions allowed and exactly
$r$ summands. Set
\[
 \tri m=\frac{m(m+1)}2,\qquad
 C^*=C\cup(-C)\cup\{0\}.
\]
A set $C$ is \emph{antisymmetric} if $C\cap(-C)=\varnothing$.
A zero-sum-free set becomes antisymmetric after all its elements of
order two have been removed.

For $B\subseteq G$, define its translation boundary by
\[
 \lambda_B(x)=|(B+x)\setminus B|.
\]
We use
\begin{equation}\label{eq:boundary-properties}
 \lambda_B(x)=\lambda_B(-x)=\lambda_{G\setminus B}(x),\qquad
 \lambda_B(x+y)\le\lambda_B(x)+\lambda_B(y).
\end{equation}
For $1\le j<|X|$, write $j\wedge X$ for the sums of $j$ distinct
elements of $X$.

The following standard estimates will be used throughout:
\begin{equation}\label{eq:basic}
 \N(A)\ge\max\{|A|^2/6,\,2|A|-1\}
 \quad\text{for every nonempty zero-sum-free set }A,
\end{equation}
and
\begin{equation}\label{eq:superadditive}
 \N(PQ)\ge\N(P)+\N(Q)
 \quad\text{if the concatenation }PQ\text{ is zero-sum-free}.
\end{equation}
See \cite[Theorem 1.1 and Lemmas 2.1--2.2]{GHHLLP}.
We set $\N(\varnothing)=0$.

We also use Kneser's theorem in its multiple-summand form: if $J$ is
the stabilizer of $A_1+\cdots+A_t$, then
\begin{equation}\label{eq:kneser}
 |A_1+\cdots+A_t|\ge\sum_{i=1}^t|A_i+J|-(t-1)|J|.
\end{equation}
By \cite[Theorem~3.2]{BGGH}, including its elementary two-group exception,
\begin{equation}\label{eq:restricted}
 |j\wedge X|\ge |X|-1\qquad(1\le j<|X|).
\end{equation}

We recall the structural input needed only for the repeated-support
lemma below. A set $X\subseteq Q$ is a \emph{Vosper set} if
$|X+Y|\ge\min\{|Q|-1,|X|+|Y|\}$ whenever $|Y|\ge2$.
An AP-representation or a Vosper-representation of $A\subseteq G$
means a subgroup $J$ for which
\[
 |A+J|<\min\{|G|,|A|+|J|\},
\]
and the image of $A$ in $G/J$ is, respectively, an arithmetic progression
or a Vosper set. If $0\in A$, $\gen A=G$, and $|A|\le|G|/2$, such a
representation exists \cite[Theorem 3.1]{BGGH}.
Call $U$ \emph{valid} if $|\Sig(V)|\le|\gen V|/2$ for every nonempty
$V\subseteq U$. The following consequence of
\cite[Proposition 4.1]{BGGH} handles the AP case: if $U$ is valid,
antisymmetric, generates $G$, $|U|=m\ge4$, and $U^*$ has an
AP-representation, then $|\Sig(U)|\ge\tri m+1$.
Here $\Sig$ includes the empty sum, as does the subset-sum notation
in \cite{BGGH}. The growth and boundary estimates for the Vosper case
are proved below under the four-copy zero-sum-free hypothesis used
in this paper.

Finally, Olson's boundary estimate, in the form
\cite[Lemma 2.7]{GHHLLP}, says that if $C$ is antisymmetric,
$c=|C|$, $H=\gen C$, and
$|H|\ge2b$ for $b=\min\{|B|,|G\setminus B|\}\ge1$, then
\begin{equation}\label{eq:olson-boundary}
 \max_{x\in C}\lambda_B(x)\ge\frac12\min\{b+1,c+1\}.
\end{equation}
In this assertion $H$ may be a proper subgroup of $G$.

\section{Repeated supports}

\begin{lemma}\label{lem:four-copy-growth}
Let $U$ be an $m$-element set generating a finite abelian group $G$,
where $m\ge4$. Suppose $U^4$ is zero-sum-free and $U^*$ has a
Vosper-representation. Then, for every integer $j\ge1$,
\begin{equation}\label{eq:four-copy-growth}
 |jU^*|\ge\min\{|G|,2mj+1\}.
\end{equation}
\end{lemma}

\begin{proof}
The four-copy assumption makes $U$ antisymmetric and valid. To see
validity, suppose $Z=\Sig(V)$ has size greater than $|\gen V|/2$
for some nonempty $V\subseteq U$. Then $Z+Z=\gen V$, and a
representation of $-v$ by $V^2$, followed by a third copy of any
$v\in V$, gives a zero-sum subsequence of $U^4$.

Write $X=U^*$, and let $H$ be the subgroup in a
Vosper-representation of $X$. Put
\[
 h=|H|,\quad Q=\pi_H(X),\quad u=|Q|,\quad v=|G/H|,
 \quad \delta=hu-|X|.
\]
The representation is strict, so $0\le\delta<h$ and $Q$ is a proper
generating subset of $G/H$. It is symmetric and contains zero.
Consequently $u\ge3$: a symmetric generating set of size two containing
zero would be the entire quotient of order two.

We first prove the kernel estimate
\begin{equation}\label{eq:four-copy-kernel}
 \delta\ge\frac{h-1}{2}.
\end{equation}
Let $A=U\cap H$, $a=|A|$. Since $|X\cap H|=2a+1$,
$\delta\ge h-2a-1$. If $a=0$, this proves
\eqref{eq:four-copy-kernel}. If $a>0$, validity and
\eqref{eq:basic} give
\[
 2a\le|\Sig(A)|\le |\gen A|/2\le h/2.
\]
For odd $h$ this yields $2a\le(h-1)/2$, as required.
For even $h$, equality $|\Sig(A)|=h/2$ is impossible. Indeed, set
$Z=\Sig(A)$ and $D=\Sig(A^2)=Z+Z$. The fact that $A^4$ is
zero-sum-free gives $D\cap(-D)=\{0\}$, hence $|D|\le h/2$.
Since $Z\subseteq D$ and $|Z|=h/2$, we obtain $Z+Z=Z$.
A finite set containing zero and closed under addition is a subgroup.
Thus $Z=\gen A$, contradicting validity. Therefore
$2a\le h/2-1$ when $h$ is even, which again proves
\eqref{eq:four-copy-kernel}.

Next we show that $2X$ is $H$-periodic. For $R\in Q$, let
$X_R=X\cap\pi_H^{-1}(R)$. The total number of missing points
in these fibers is $\delta<h$. If $R\ne S$, then
$|X_R|+|X_S|>h$, so $X_R+X_S$ fills the whole coset indexed by
$R+S$. If $R\ne-R$, symmetry gives $|X_R|=|X_{-R}|>h/2$;
thus $X_R+X_R$ also fills the coset indexed by $2R$.
The remaining repeated-fiber sums have $R=-R$ and target the
zero coset. This coset is filled by $X_R+X_{-R}$ if a non-self-inverse
direction exists. Otherwise $Q\setminus\{0\}$ contains at least
two self-inverse directions. At most one of their fibers has size at
most $h/2$, since their total deficit is less than $h$. A fiber of
size greater than $h/2$ then fills $H$ by addition to itself.
Every coset of $2Q$ is therefore full in $2X$. It follows that
\begin{equation}\label{eq:periodic-iterate}
 |jX|=h|jQ|\qquad(j\ge2).
\end{equation}

Let $t$ be the last integer for which $tQ\ne G/H$. Such an integer
exists because $Q$ is proper, contains zero, and generates the
quotient. If $t=1$, \eqref{eq:four-copy-growth} follows at once from
$|X|=2m+1$ and \eqref{eq:periodic-iterate}.
Suppose $t\ge2$. The Vosper property gives
\begin{equation}\label{eq:quotient-last-layer}
 |jQ|\ge ju\quad(1\le j<t),\qquad
 |tQ|\ge tu-1,\qquad tu\le v.
\end{equation}
For completeness, induct on $j<t$. The only alternative to
$|jQ|\ge ju$ in the Vosper bound is $|jQ|=v-1$; then
$jQ+Q=G/H$, contradicting $j+1\le t$. At the last layer,
the same bound gives $|tQ|\ge\min\{v-1,tu\}$. If $tu>v$,
then $|(t-1)Q|+|Q|>v$ would force $tQ=G/H$, also a
contradiction. This proves \eqref{eq:quotient-last-layer}.

For $2\le j<t$, equations \eqref{eq:periodic-iterate} and
\eqref{eq:quotient-last-layer} give $|jX|\ge j|X|$.
The same holds for $j=t$ if $tu\le v-1$. In the only other case,
$tu=v$, equations \eqref{eq:four-copy-kernel} and
\eqref{eq:periodic-iterate} yield
\[
 |tX|\ge h(tu-1)=t|X|+t\delta-h\ge t|X|-1
 \ge2mt+1.
\]
Here $t\ge2$ and $t(h-1)/2-h\ge-1$.
The first layer has size $2m+1$, and the layers after $t$ are full.
This proves \eqref{eq:four-copy-growth}.
\end{proof}

\begin{lemma}\label{lem:four-copy-boundary}
Under the hypotheses of Lemma~\ref{lem:four-copy-growth}, let
$B\subseteq G$ satisfy
\[
 m(m-1)<b=|B|\le |G|/2.
\]
Then
\begin{equation}\label{eq:usable-vosper-boundary}
 \max_{x\in U}\lambda_B(x)>
 m-\frac{m(m-1)}b>m-1.
\end{equation}
In particular, some $x\in U$ has $\lambda_B(x)\ge m$.
\end{lemma}

\begin{proof}
Put $\alpha=\max_{x\in U}\lambda_B(x)$ and $\ell=2b-3$.
We have $1\le\ell\le |G|-1$. Write
$\ell=2mR+q$, where $0\le q<2m$. The growth estimate
\eqref{eq:four-copy-growth} permits us to choose pairwise disjoint
sets $E_j\subseteq jU^*\setminus\{0\}$ with
\[
 |E_j|=2m\ (1\le j\le R),\qquad |E_{R+1}|=q.
\]
Indeed, each ball contains enough nonzero points to add the next
required set. Let $E$ be their union. Counting ordered differences
of distinct elements of $B$ gives
\[
 \sum_{x\in E}\lambda_B(x)
 \ge\ell b-b(b-1)=b(\ell-b+1).
\]
On the other hand, subadditivity of the boundary gives
\begin{align*}
 \sum_{x\in E}\lambda_B(x)
 &\le\alpha\left(2m\sum_{j=1}^Rj+q(R+1)\right)\\
 &=\alpha(R+1)(mR+q)
 \le\frac{\alpha(\ell+m)^2}{4m}.
\end{align*}
The last inequality follows by completing the square in $q$.
Consequently
\[
 \alpha\ge\frac{4mb(b-2)}{(2b+m-3)^2}
 =\frac{m(1-2/b)}{(1+(m-3)/(2b))^2}.
\]
Since $m\ge4$ and $b>m(m-1)$, we may use
$(1+x)^{-2}>1-2x$ for $x>0$ to obtain the following strict bound.
Integrality of $\alpha$ then gives the final assertion.
\begin{align*}
 \alpha
 &>m\left(1-\frac2b\right)\left(1-\frac{m-3}b\right)\\
 &=m-\frac{m(m-1)}b+\frac{2m(m-3)}{b^2}
 >m-\frac{m(m-1)}b>m-1.\qedhere
\end{align*}
\end{proof}

\smallskip
\noindent\textit{Remark.}
Our argument follows \cite[Section~4.2]{BGGH}. In
\href{https://arxiv.org/abs/1112.1929v1}{arXiv:1112.1929v1},
the final comparison in the proof of Lemma~4.6 (p.~18) is incorrect:
in our notation, the product's $1/b$ term is $-m(m-1)/b$, as expanded
above. For our application it suffices that the boundary exceed $m-1$: when $b>m(m-1)$ the corrected estimate gives this strict inequality, and the boundary is an integer. Lemmas~\ref{lem:four-copy-growth}--\ref{lem:four-copy-boundary}
establish the estimates needed here under the four-copy hypothesis
without invoking \cite[Lemma~4.2]{BGGH}.

\begin{lemma}\label{lem:repetition}
Let $r\ge2$, and let $U$ be an $m$-element set in a finite abelian group.
If $U^{2r}$ is zero-sum-free, then
\begin{equation}\label{eq:repetition}
 \N(U^r)\ge r\tri m.
\end{equation}
\end{lemma}

\begin{proof}
Fix $r$ and induct on $m$, replacing the ambient group by $\gen U$.
For $m\le2$, $\N(U)\ge\tri m$ is immediate. For $m=3$, the three
one-element sums are distinct, as are the three two-element sums.
There is at most one equality between the two families: two such
equalities would force an element of $U$ to have order two. The total
sum belongs to neither family. Hence $\N(U)\ge6$.
Equation~\eqref{eq:superadditive} proves \eqref{eq:repetition} in these
base cases.

Assume $m\ge4$. The set $U$ is antisymmetric. It is also valid.
Indeed, if $A=\Sig(V)$ had size greater than $|\gen V|/2$ for some
nonempty $V\subseteq U$, then $A+A=\gen V$. A representation of
$-v$ by $V^2$, followed by a third copy of any $v\in V$, would give a
nonempty zero-sum subsequence of $U^{2r}$.

Put $B=\Sig(U^r)$, $b=|B|$, and $n=|G|$. First suppose $b>n/2$.
Since $U^{2r}$ is zero-sum-free,
$B\cap(-B)=\{0\}$. Therefore $n=2b-1$ is odd. Let
$c=\sigma(U^r)$ and $M=b-1$, so $n=2M+1$. Taking complementary
subsequences gives $B=c-B$, and consequently
\begin{equation}\label{eq:dense-alternation}
 B\cup(B-c)=G,\qquad B\cap(B-c)=\{0\}.
\end{equation}
On every orbit of translation by $c$ not containing zero, membership
in $B$ would alternate. Such an orbit has odd length, which is
impossible. Thus $c$ generates $G$. On its one orbit,
\eqref{eq:dense-alternation} determines $B$: with $g=-2c$,
\[
 B=\{0,g,\ldots,Mg\},\qquad c=Mg.
\]
Write the distinct elements of $U$ as $w_1g,\ldots,w_mg$, where
$1\le w_i\le M$. If the integer sum of the weights of $U^r$ first
exceeded $M$ at some prefix, that prefix weight would lie in
$[M+1,2M]$, contradicting the description of $B$. Thus
$r\sum_iw_i\le M$. Its residue is $M$ modulo $2M+1$, so equality
holds. It follows that $\N(U^r)=M\ge r\tri m$.

Now suppose $b\le n/2$. Every zero-sum-free sequence has at least
its length in nonempty sums, so $b\ge rm+1\ge2m+1=|U^*|$.
The representation theorem applies to $U^*$. In the AP case, validity
and the input from Section~2 give $\N(U)\ge\tri m$, and
\eqref{eq:superadditive} finishes the proof.

In the Vosper case, the induction hypothesis gives
\[
 b\ge r\tri{m-1}+1\ge m(m-1)+1.
\]
Also $U^4$ is zero-sum-free, since $r\ge2$.
Thus Lemma~\ref{lem:four-copy-boundary} applies and supplies an
$x\in U$ with $\lambda_B(x)\ge m$. For $0\le j\le r+1$, put
\[
 B_j=\Sig((U\setminus\{x\})^r)+\{0,x,\ldots,jx\}.
\]
The increments $d_j=|B_j\setminus B_{j-1}|$ are nonincreasing:
if $y\in B_{j+1}\setminus B_j$, then $y-x\in B_j\setminus B_{j-1}$.
Indeed, $y$ must come from the new translate, and membership of
$y-x$ in $B_{j-1}$ would imply $y\in B_j$.
Since $B_r=B$ and $d_{r+1}=\lambda_B(x)$, induction yields
\[
 \N(U^r)\ge\N((U\setminus\{x\})^r)+r\lambda_B(x)
 \ge r\tri{m-1}+rm=r\tri m.\qedhere
\]
\end{proof}

\Needspace{14\baselineskip}
\section{Quotient sums and almost full fibers}

The next lemma permits zero relations at the corners of a box of
multiplicities. It excludes them only when some coordinate is strictly
between its two endpoints.

\begin{lemma}\label{lem:capacity}
Let $q_1,\ldots,q_m$ be distinct elements of a finite abelian group,
and let $7\le L_i\le h$ be integers. Suppose
\begin{equation}\label{eq:internal-nonzero}
 \sum_i j_iq_i\ne0
 \quad\text{whenever }0\le j_i\le L_i
 \text{ and some }0<j_i<L_i.
\end{equation}
Let $\mathcal I$ be the set of sums specified in
\eqref{eq:internal-nonzero}, and put $e=\sum_i(h-L_i)$. Then
\begin{equation}\label{eq:capacity}
 |\mathcal I|\ge h\tri m-4me-1.
\end{equation}
\end{lemma}

\begin{proof}
The case $m=0$ is immediate, so assume $m\ge1$.
Let $O=\{i:L_i\text{ is odd}\}$, $o=|O|$, and set
\[
 b_0=\begin{cases}h,&h\text{ even},\\h-3,&h\text{ odd},\end{cases}
 \qquad E_i=\min\{b_0,L_i-3\boldsymbol1_{i\in O}\}.
\]
Each $E_i$ is even and at least four. Form an ordered list of sequences:
first three copies of $\{q_i:i\in O\}$, if $O\ne\varnothing$,
and then, for $t=1,\ldots,\max_i E_i/2$, two copies of
\[
 V_t=\{q_i:E_i\ge2t\}.
\]
The sets $V_t$ are nested, and $V_1=V_2=\{q_1,\ldots,q_m\}$.
The total multiplicity used in direction $i$ is
$E_i+3\boldsymbol1_{i\in O}\le L_i$.

For each sequence in this list take its set of sums, including zero,
and translate it by the total sum of all preceding sequences. Consecutive
translated sets have a common endpoint. We check all other intersections.
If a point from a later block equals a point from an earlier block,
subtracting their multiplicity vectors gives $w$ with
$0\le w_i\le L_i$ and $\sum_iw_iq_i=0$. By
\eqref{eq:internal-nonzero}, every $w_i$ is either $0$ or $L_i$.

For consecutive blocks, $w_i\le5<L_i$; hence $w=0$, and the
intersection is precisely their common endpoint. For nonconsecutive
two-copy blocks, if the earlier is not the first two-copy block,
choose a direction active in the later block. It has already occurred
before the earlier block and also occurs in a block strictly between
the two. Thus $0<w_i<L_i$, a contradiction.
If the earlier block is the first two-copy block, the intervening
second two-copy block makes every $w_i$ positive. An odd direction,
if present, has already used three copies and gives $w_i<L_i$.
If no odd direction is present, all $w_i=L_i$, which is possible
only at the global initial and final endpoints. Finally, between the
initial three-copy block and a nonadjacent block, the first two-copy
block again makes all coordinates positive; only the same global
endpoint collision can remain. No intermediate endpoint participates
in this collision, because its multiplicity vector has an internal
coordinate.

The six-copy sequence on the whole support is zero-sum-free, since
$6<L_i$ for all $i$. Lemma~\ref{lem:repetition}, with $r=3$ for
the initial block and $r=2$ for the other blocks, therefore applies.
Let $n_j$ count the nonempty sums of block $j$. Adjacent overlaps
leave $1+\sum_j n_j$ points if the global endpoints are distinct;
deleting both leaves $\sum_j n_j-1$. If they coincide, the union has
one fewer point but only one endpoint is deleted, giving the same
bound. All remaining multiplicity vectors have an internal coordinate.
The repeated-support bounds therefore give
\begin{equation}\label{eq:block-count}
 |\mathcal I|\ge3\tri o+2\sum_t\tri{|V_t|}-1.
\end{equation}

To estimate the right side, note that removing one occurrence of a
direction from a layer decreases its triangular count by at most $m$.
If $h$ is even, $o\le e$ and $\sum_i(h-E_i)=e+3o$, whence
\[
 3\tri o+2\sum_t\tri{|V_t|}
 \ge h\tri m-me-3mo+3\tri o\ge h\tri m-4me.
\]
If $h$ is odd, $m-o\le e$ and
$\sum_i((h-3)-E_i)\le e$. Hence
\[
 3\tri o+2\sum_t\tri{|V_t|}
 \ge(h-3)\tri m-me+3\tri o
 \ge h\tri m-4me,
\]
using $\tri m-\tri o\le m(m-o)$. Combining these inequalities with
\eqref{eq:block-count} proves the lemma.
\end{proof}

\begin{lemma}\label{lem:opposite}
Let $C$ be antisymmetric and zero-sum-free, let $K$ be a subgroup of
order $h$, and put $A=C\cap K$, $a=|A|\ge2$. Write
\[
 C\cap(g+K)=g+X,\qquad C\cap(-g+K)=-g+Y
\]
for two distinct opposite cosets. If $X,Y$ are both nonempty, then
\begin{equation}\label{eq:opposite}
 h-|X|-|Y|\ge a-1.
\end{equation}
\end{lemma}

\begin{proof}
Antisymmetry gives $X\cap(-Y)=\varnothing$. Suppose
$e=h-|X|-|Y|<a-1$. Set $Z=\Sig(A)$, so $|Z|\ge2a$ by
\eqref{eq:basic}. The set $D=Z+X+Y$ avoids zero. For
$J=\stab_K(D)$ and $j=|J|$, Kneser's theorem gives
\[
 |Z+J|+|X+J|+|Y+J|-2j\le h-j,
 \qquad |Z+J|\le e+j.
\]
If $j\le e$, then $|Z|\le2e<2a$, a contradiction. Thus $j>e$,
and divisibility by $j$ gives $Z\subseteq J$ and
$|X+J|+|Y+J|=h$.
Since $D+J=D$ avoids zero and $0\in Z$, we have
$(X+J)\cap(-Y+J)=\varnothing$. These sets therefore partition $K$.
In $L=K/J$ this gives $\bar Y=-(L\setminus\bar X)$. For $u\in L$,
\[
 u\notin\bar X+\bar Y\quad\Longleftrightarrow\quad
 \bar X=u+\bar X.
\]
Thus the complement of $\bar X+\bar Y$ is the subgroup
$P=\stab_L(\bar X)$. This subgroup also stabilizes
$L\setminus P=\bar X+\bar Y$. Since $J$ was the full stabilizer
of $D$, the quotient sumset is aperiodic, so $P$ is trivial.
In particular, $\bar X+\bar Y=L\setminus\{0\}$.

Choose a nonzero $u\in L$ and representations
$\bar x_1+\bar y_1=u$, $\bar x_2+\bar y_2=-u$.
Fix actual $x_2,y_1$, and vary $x_1,y_2$ in their indicated
$J$-fibers. These two fibers together miss at most $e$ points,
so at least $j-e$ pairs solve $x_1+x_2+y_1+y_2=0$ in $K$.
Since $j\ge|Z|\ge2a$ and $e<a-1$, more than two solutions exist.
Excluding $x_1=x_2$ and $y_1=y_2$ eliminates at most two. A
remaining solution gives four distinct elements of $C$ summing to
zero, a contradiction.
\end{proof}

\begin{lemma}\label{lem:fibers}
Let $C$ be antisymmetric and zero-sum-free, $K$ a subgroup of order
$h$, and $A=C\cap K$ with $a=|A|\ge2$. Suppose its occupied
nonzero $K$-cosets are $q_1,\ldots,q_m$, with no self-inverse
coset and no opposite pair. If their fiber sizes $L_i$ satisfy
$L_i\ge7$ and $e=\sum_i(h-L_i)<a-1$, then
\begin{equation}\label{eq:fibers}
 \N(C)\ge h\bigl(h\tri m-4me-1\bigr)+\N(A).
\end{equation}
\end{lemma}

\begin{proof}
For $Z=\Sig(A)$, \eqref{eq:basic} gives $|Z|\ge2a>e+1$.
If $X_i$ is a fiber translated into $K$, then
\eqref{eq:restricted} and $h-L_i\le e$ imply
\[
 Z+(j\wedge X_i)=K\qquad(0<j<L_i).
\]
Consequently every multiplicity vector having an internal coordinate
fills the entire coset of its quotient sum. That quotient sum cannot
be zero, since all these sums use an external element. Thus
\eqref{eq:internal-nonzero} holds. Each point of $\mathcal I$ contributes
$h$ nonempty sums outside $K$, and $A$ contributes $\N(A)$ more.
Apply Lemma~\ref{lem:capacity}.
\end{proof}

\section{An expansion alternative}

\begin{proposition}\label{prop:alternative}
Let $C$ be an antisymmetric zero-sum-free set, $c=|C|$, and
$H=\gen C$. Then at least one of the following holds:
\begin{enumerate}
\item $\N(C)\ge c^2/2$;
\item for every integer $r\ge1$ for which $rC^*\ne H$,
\begin{equation}\label{eq:expansion}
 |rC^*|\ge2r(c-100).
\end{equation}
\end{enumerate}
\end{proposition}

\begin{proof}
There is nothing to prove if $c\le100$. Suppose
$rC^*\ne H$ and $|rC^*|<2r(c-100)$. Put
$K=\stab_H(rC^*)$, $h=|K|$, $A=C\cap K$, and $a=|A|$.
Let $p$ count the nonzero $K$-cosets met by $C^*$.
Kneser's theorem gives
\begin{equation}\label{eq:expansion-kn}
 |rC^*|\ge(rp+1)h.
\end{equation}
Here $p\ge2$. Indeed, $p=0$ would imply $K=H$; and if $p=1$,
the sole nonzero quotient element is self-inverse and generates
$H/K$, forcing $rC^*=H$.
Antisymmetry also gives $2(c-a)\le ph$. Thus, writing
\[
 e=\frac{ph}{2}-(c-a)\ge0,
\]
the assumed failure and \eqref{eq:expansion-kn} imply
\begin{equation}\label{eq:kernel-budget}
 a>100,\qquad e<a-100,\qquad a^2<6h,
 \qquad h>1600,\qquad a<h/16.
\end{equation}
The third inequality follows from $\N(A)+1\le h$ and
\eqref{eq:basic}; the last two follow from $a>100$ and $a^2<6h$.

No distinct opposite pair of external cosets can both be occupied.
Otherwise its contribution to $e$ is at least $a-1$ by
Lemma~\ref{lem:opposite}, contradicting \eqref{eq:kernel-budget}.
There is no occupied self-inverse nonzero coset either. If one has
size $L$, its elements lie in a subgroup of order $2h$, so
\eqref{eq:basic} gives $L<\sqrt{12h}$. Its contribution to $e$ is
$h/2-L$, whereas $e<a$ gives
\[
 L>h/2-a>h/2-\sqrt{6h}>\sqrt{12h}
 \quad(h>1600),
\]
a contradiction.

We may therefore write $p=2m$, with one occupied coset per opposite
pair. If their sizes are $L_i$, then
\[
 e=\sum_i(h-L_i),\qquad c=mh+a-e,\qquad
 L_i\ge h-e>h-a+100>7.
\]
Lemma~\ref{lem:fibers} applies. After dropping its nonnegative
$\N(A)$ term, we obtain
\begin{align*}
 \N(C)-\frac{c^2}{2}
 &\ge \frac{mh^2}{2}-mha-3mhe-\frac{(a-e)^2}{2}-h\\
 &>\frac{mh^2}{2}-4mha-4h
   =\frac{mh}{2}\left(h-8a-\frac8m\right)>0.
\end{align*}
The last inequality uses $a<h/16$ and $h>1600$.
Thus every failure of \eqref{eq:expansion} gives the first alternative.
\end{proof}

\section{From expansion to subset sums}

\begin{lemma}\label{lem:exp-boundary}
Let $C\subseteq G$, $H=\gen C$, and $d>0$. Suppose, for every
$r\ge1$, either $rC^*=H$ or $|rC^*|\ge2rd$.
If $b=\min\{|B|,|G\setminus B|\}>d$ and $|H|\ge2b$, then
\begin{equation}\label{eq:exp-boundary}
 \max_{x\in C}\lambda_B(x)>d-d^2/b.
\end{equation}
\end{lemma}

\begin{proof}
By complementation assume $|B|=b$. Suppose all boundaries are at
most $M=d-d^2/b>0$. Put $\rho(x)=|B\cap(B+x)|$.
For $x\in rC^*$, subadditivity gives $\rho(x)\ge b-rM$, whereas
\begin{equation}\label{eq:correlation}
 \sum_{x\in H}\rho(x)=\sum_{Q\in G/H}|B\cap Q|^2\le b^2.
\end{equation}
Write $x_0=b/d=R+\theta$, with $R=\lfloor x_0\rfloor$ and
$0\le\theta<1$. The hypotheses imply
$|rC^*|\ge2rd$ for $1\le r\le R$ and
$|(R+1)C^*|\ge2b$. Also $t_0=b-(R+1)M>0$.
Integrating the level sets of $\rho$ over the consecutive intervals
of length $M$, and then over $[0,t_0]$, yields
\[
 \sum_{x\in H}\rho(x)\ge dMR(R+1)+2bt_0.
\]
The right side, divided by $b^2$, equals
\[
 1+\frac{1-\theta+\theta^2+\theta(1-\theta)/x_0}{x_0^2}>1,
\]
contradicting \eqref{eq:correlation}.
\end{proof}

We now prove the global estimate for antisymmetric sets. For $s\ge1$
write
\[
 L(s)=\lceil\log_{3/2}s\rceil,\qquad
 \calH_s=\sum_{j=1}^s\frac1j.
\]

\begin{lemma}\label{lem:quarter}
Every antisymmetric zero-sum-free set $T$ of size $r\ge1$ satisfies
\begin{equation}\label{eq:quarter}
 \N(T)\ge\frac{r^2}{4}-\frac{r+1}{2}L(r)-1.
\end{equation}
\end{lemma}

\begin{proof}
Arrange $T=\{a_1,\ldots,a_r\}$ greedily, always choosing the largest
boundary relative to the current prefix sumset. Let
$B_j=\Sig(\{a_1,\ldots,a_j\})$ and $\sigma_j=|B_j|$, with
$\sigma_0=1$. Before step $i$, set
$C_i=\{a_i,\ldots,a_r\}$, $c_i=r-i+1$, and $H_i=\gen{C_i}$.
Stop before the first step for which
\begin{equation}\label{eq:capacity-stop}
 |H_i|<2\min\{\sigma_{i-1},|G\setminus B_{i-1}|\},
 \qquad G=\gen T.
\end{equation}
Before stopping, zero-sum-freeness and \eqref{eq:basic} imply
$|G\setminus B_{i-1}|\ge\N(C_i)\ge2c_i-1$. Hence
\eqref{eq:olson-boundary} gives
\begin{equation}\label{eq:weak-growth}
 \sigma_i-\sigma_{i-1}\ge\frac12\min\{\sigma_{i-1}+1,c_i+1\}.
\end{equation}
While $\sigma_{i-1}<c_i$, the quantity $\sigma_i+1$ increases by
a factor at least $3/2$. This regime has ended by $L(r)$ steps if
the process has not already stopped. Once it ends, it does not
restart. Summing the lower bounds $(c_i+1)/2$ thereafter, and
discarding the first $L(r)$ ideal increments, gives, for each
available prefix,
\begin{equation}\label{eq:quarter-prefix}
 \sigma_j\ge\frac{r^2-(r-j)^2}{4}-\frac{r+1}{2}L(r).
\end{equation}
For $j\le L(r)$ the right side is nonpositive, so it also covers
an early stop.

If there is no stop, subtract one from \eqref{eq:quarter-prefix}
at $j=r$. Otherwise the stopping subgroup $H$ is proper in $G$.
Let $v$ be the last position with $a_v\notin H$, and put
$T'=\{a_{v+1},\ldots,a_r\}$, $t=r-v$. Both $v$ and $v-1$
precede the stopping position. Since $T'\subseteq H$ and $a_v\notin H$,
the two parts of
\[
 \Sig(a_vT')=\Sig(T')\sqcup\bigl(a_v+\Sig(T')\bigr)
\]
lie in distinct $H$-cosets. Hence $\N(a_vT')=2\N(T')+1$, and
\eqref{eq:superadditive} gives
\[
 \N(T)\ge\sigma_{v-1}+2\N(T')
 \ge\frac{r^2-(t+1)^2}{4}-\frac{r+1}{2}L(r)+\frac{t^2}{3}.
\]
Finally, $t^2/3-(t+1)^2/4=(t-3)^2/12-1\ge-1$.
\end{proof}

\begin{proposition}\label{prop:antisymmetric}
Every antisymmetric zero-sum-free set $S$ of size $s\ge1$ satisfies
\begin{equation}\label{eq:antisymmetric}
 \N(S)\ge\frac{s^2}{2}-10s\log(s+1)-116s.
\end{equation}
\end{proposition}

\begin{proof}
Use the same greedy notation, now in $G=\gen S$. Stop before the
first step $i$ for which either
\begin{equation}\label{eq:two-stops}
 \N(C_i)\ge c_i^2/2
 \quad\text{or}\quad
 |H_i|<2\min\{\sigma_{i-1},|G\setminus B_{i-1}|\}.
\end{equation}
Give the first condition priority if both hold. Before stopping,
\begin{equation}\label{eq:complement}
 |G\setminus B_{i-1}|\ge\N(C_i)
 \ge\max\{2c_i-1,c_i^2/6\}.
\end{equation}
The first inequality follows because the prefix sumset avoids the
negative nonempty tail sums. Thus \eqref{eq:weak-growth} remains valid.

Let $c=c_i$, $\sigma=\sigma_{i-1}$, and
$b=\min\{\sigma,|G\setminus B_{i-1}|\}$. By
Proposition~\ref{prop:alternative} and Lemma~\ref{lem:exp-boundary},
with $d=c-100$, the actual increment satisfies
\begin{equation}\label{eq:strong-growth}
 \sigma_i-\sigma_{i-1}\ge c-106-\frac{c^2}{\sigma}.
\end{equation}
For clarity, if $d\le0$ this is automatic. If $d>0$ and $b\le d$,
\eqref{eq:complement} forces $b=\sigma$, and the right side is
negative. Otherwise Lemma~\ref{lem:exp-boundary} applies. When
$b=\sigma$, it gives the stronger $c-100-c^2/\sigma$; when
$b=|G\setminus B_{i-1}|$, its error $d^2/b$ is at most six by
\eqref{eq:complement}. This proves \eqref{eq:strong-growth} in
every case.

Put $L=L(s)$. After the first $L+1$ steps, weak growth gives, for
each available $i\ge L+2$,
\begin{equation}\label{eq:harmonic}
 \sigma_{i-1}\ge\frac{(i-1-L)(c_i+2)}2,
 \qquad \frac{c_i^2}{\sigma_{i-1}}
       \le\frac{2s}{i-1-L}.
\end{equation}
Indeed, by step $L$ the initial multiplicative regime has ended.
For $L+1\le j\le i-1$, the weak increment at step $j$ is at
least $(c_j+1)/2\ge(c_i+2)/2$. Summing these $i-1-L$
increments gives the first inequality. If $L+2>s$, there are no
such steps.
Discard the first $L+1$ ideal increments at cost at most $s(L+1)$,
and sum \eqref{eq:strong-growth} and \eqref{eq:harmonic}. Since
\[
 \sum_{i=1}^j c_i\ge\frac{s^2-(s-j)^2}{2},
\]
every prefix formed before stopping obeys
\begin{equation}\label{eq:half-prefix}
 \sigma_j\ge\frac{s^2-(s-j)^2}{2}
       -s\bigl(L+107+2\calH_s\bigr).
\end{equation}
Its right side is nonpositive during the discarded initial steps,
so an early stop is included.

If the first condition in \eqref{eq:two-stops} occurs, apply
\eqref{eq:superadditive} to prefix and tail. The tail contribution
$c_i^2/2$ cancels the missing square in \eqref{eq:half-prefix}.
If no stop occurs, use $j=s$ and subtract the empty sum.

For a stop caused only by subgroup capacity, let $H$ be the stopping
subgroup, $v$ the last position outside $H$, and $T=S_{>v}$,
$r=s-v$. As in Lemma~\ref{lem:quarter},
\[
 \N(S)\ge\sigma_{v-1}+2\N(T).
\]
Apply \eqref{eq:half-prefix} at $v-1$ and
\eqref{eq:quarter} to $T$. For $r=0$ use $\N(T)=0$ directly.
Since $r+1\le s$ and $L(r)\le L(s)$ when $r\ge1$, all cases give
the convenient common bound
\begin{equation}\label{eq:explicit-transfer}
 \N(S)\ge\frac{s^2}{2}-108s-(2s+1)L(s)
                        -2s\calH_s-\frac52.
\end{equation}
To simplify it, use
$L(s)\le\tfrac52\log(s+1)+1$ and
$\calH_s\le\log(s+1)+1$. The error in
\eqref{eq:explicit-transfer} is at most
\[
 (7s+\tfrac52)\log(s+1)+112s+\tfrac72
 \le10s\log(s+1)+116s,
\]
because $(3s-\tfrac52)\log(s+1)+4s-\tfrac72\ge0$ for $s\ge1$.
This proves \eqref{eq:antisymmetric}.
\end{proof}

\section{The quantitative bound and the Olson constant}

\begin{proof}[Proof of Theorem~\ref{thm:main}]
Let $I$ be the set of all order-two elements in $S$, let $q=|I|$,
and set $X=S\setminus I$, $s=k-q$. Zero-sum-freeness makes $I$
linearly independent over $\mathbb F_2$. Therefore
\begin{equation}\label{eq:involutions}
 2^q=|\Sig(I)|\le\N(S)+1.
\end{equation}
If $\N(S)\ge k^2/2$ there is nothing to prove. Otherwise
\[
 q\le\log_2(\N(S)+1)<3\log(k+1),
\]
since $\N(S)+1<k^2/2+1\le(k+1)^2$ and $2/\log2<3$.
The set $X$ is antisymmetric. Proposition~\ref{prop:antisymmetric}
and $s^2/2\ge k^2/2-kq$ yield
\begin{align*}
 \N(S)&\ge\N(X)\\
 &\ge\frac{k^2}{2}-kq-10s\log(s+1)-116s\\
 &\ge\frac{k^2}{2}-13k\log(k+1)-116k.
\end{align*}
If $s=0$, use $\N(X)=0$ in the same estimate; its right side is
then nonpositive. This also covers all small values of $k$.
\end{proof}

\begin{proof}[Proof of Corollary~\ref{cor:olson}]
Let $S$ be a zero-sum-free set of maximum size $k=\Ol(G)-1$.
The case $k=0$ is immediate. Since $\N(S)\le n-1$,
\eqref{eq:basic} gives $k\le\sqrt{6n}$. Put
$a=13\log(k+1)+116$. Theorem~\ref{thm:main} implies
$k^2/2-ak\le n-1$, whence
\[
 k\le a+\sqrt{a^2+2(n-1)}\le\sqrt{2n}+2a.
\]
Also $k+1\le(\sqrt6+1)\sqrt{n+1}$. It follows that
\[
 \Ol(G)\le\sqrt{2n}+13\log(n+1)
                   +26\log(\sqrt6+1)+233
 <\sqrt{2n}+13\log(n+1)+266.
\]
\end{proof}

The coefficient $\sqrt2$ is necessary uniformly over finite abelian
groups. For $n>\tri t$, the residues $1,2,\ldots,t$ form a
zero-sum-free set in $\mathbb Z/n\mathbb Z$. Taking
$n=\tri t+1$ shows that $\Ol(\mathbb Z/n\mathbb Z)\ge t+1$
along a sequence with $t\sim\sqrt{2n}$.

\stepcounter{theorem}
The exact value of $f(k)$ also gives an eventual bound on the Olson constant.

\begin{corollary}
For every sufficiently large $n$ and every abelian group $G$ of order $n$,
\[
  \Ol(G)\le \left\lceil\sqrt{2(n-1)}\right\rceil .
\]
\end{corollary}

\begin{proof}
Let $k=\Ol(G)-1$ and take a zero-sum-free set of size $k$ in $G$.
If $k\ge k_0$, Theorem~\ref{th:exact} gives
$\lfloor k^2/2\rfloor+1\le n-1$, hence $k<\sqrt{2(n-1)}$.
For $k<k_0$ the same strict inequality holds once $n$ is sufficiently large.
\end{proof}

Suppose $G$ has a subgroup $H=\{0,h\}$ with cyclic quotient of order $q$.
This includes both $C_{2q}$ and $C_2\oplus C_q$; the latter is noncyclic
when $q$ is even. We show that the bound is attained for all sufficiently
large $q$. Set $k=\lceil\sqrt{4q-2}\rceil-1$ and choose a generator $g$
of $G/H$. Take $h$ together with the first $k-1$ elements obtained by
listing both elements of each coset $jg$ in order of increasing
$j=1,2,\ldots,q-1$. The sum of the positive indices used, counted with
multiplicity, is $m^2$ if $k=2m$ and $m(m+1)$ if $k=2m+1$.
In either case it is $\lfloor k^2/4\rfloor<q$. Thus every subset meeting
$G\setminus H$ has nonzero image in $G/H$, and the only other nonempty
subset is $\{h\}$. This gives a zero-sum-free set of size $k$, so the
corollary yields $\Ol(G)=k+1=\lceil\sqrt{2(|G|-1)}\rceil$.

\section{A progression of full cosets}\label{sec:progression}

We work under the bound \(M(S)\le T_k+1\), with \(G=\langle S\rangle\), along any sequence \(k\to\infty\). Constants in this proof are absolute. Every proper subgroup \(J<G\) contains fewer than \(4k/5\) terms for sufficiently large \(k\). Otherwise one exterior term supplies a disjoint translate of \(\Sig(S\cap J)\). By Theorem~\ref{thm:main}, the uniform bound \(\N(B)\ge |B|^2/2-O(|B|\log(|B|+1))\) gives \(M(S)\ge(16/25-o(1))k^2>T_k+1\).

For a finite nonempty set \(X\subseteq G\), set
\[
             \delta_X(z)=|(X+z)\setminus X|/|X|.
\]
This is symmetric under \(z\mapsto-z\) and subadditive, by the triangle inequality for symmetric differences. Also \(\sum_{z\in G}|X\cap(X+z)|=|X|^2\).

We use the iterated sumset theorem of Tao and Vu \cite[Theorem~1.21]{TV} with \(d=2\). In the form needed here there are absolute constants \(C_0,C_1\), with \(C_1\) an integer after enlargement, such that if \(m^2|A|\ge C_0|mA|\), then for some symmetric proper coset progression \(Q\) of rank at most one,
\[
                      x+Q\subseteq mA\subseteq x'+C_1Q.
\]
Here \(mA\) denotes ordinary iterated addition with repetitions. Set \(b_0>10(C_1+1)\) to be an integer, choose \(0<\eta<1/(100b_0)\), and take \(m=\lceil C_2\sqrt k\rceil\) with \(C_2\) a sufficiently large absolute constant in terms of \(C_0\).

Build \(U_1\subset S\) greedily. Start with \(X=\{0\}\); whenever an unused \(s\) satisfies \(\delta_X(s)>\eta/m\), add it to \(U_1\) and replace \(X\) by \(X\cup(X+s)\). Thus \(X=\Sigma_0(U_1)\) throughout. Its size grows by a factor greater than \(1+\eta/m\) at every step and is bounded by \(T_k+1\), so at most \(O(\sqrt k\log k)\) positions are taken. Stop when no qualifying position remains. Repeat from scratch using \(S\setminus U_1\) to obtain \(U_2\) disjoint from \(U_1\). Put \(X_i=\Sigma_0(U_i)\), \(A=\{0\}\cup (S\setminus(U_1\cup U_2))\). Then \(|A|=k-O(\sqrt k\log k)\), and \(\delta_{X_i}(a)\le\eta/m\) for every \(a\in A\), for both \(i\). In particular every \(z\in mA\) satisfies \(|X_i\cap(X_i+z)|\ge(1-\eta)|X_i|\). Summing over \(z\in mA\) gives \(|mA|\le |X_i|/(1-\eta)\). Thus the hypothesis of the iterated sumset theorem holds by the choice of \(C_2\).

Take its \(Q\) and let \(B=b_0Q\). Since \(0\in Q\) and \(x+Q\subseteq mA\), we have \(Q\subseteq mA-mA\). Consequently \(\delta_{X_i}(z)\le2b_0\eta\) for every \(z\in B\). Averaging \(\#\{z\in B:u+z\in X_i\}\) over \(u\in X_i\) gives some \(u_i\in X_i\) for which more than \(0.98|B|\) points of \(u_i+B\) lie in \(X_i\). Meanwhile \(A\subseteq 2C_1Q\), since \(0\in A\subseteq mA\) and \(Q\) is symmetric.

Write \(Q=H+[-w,w]g\) for an integer \(w\ge1\), or \(Q=H\) in rank zero. Put \(J=\langle Q\rangle\). It contains \(A\), hence \(J=G\) by the proper-subgroup observation. In rank zero \(B=G\) and the two subsets of density above \(0.98\) have sumset \(G\). This is impossible: \(X_1+X_2=\Sigma_0(U_1\cup U_2)\) omits the negative of any remaining term by freeness. The same argument excludes \(B=G\) in rank one. Hence \(G/H\) is cyclic of order \(q>2b_0w+1\), generated by \(\bar g\), and
\[
                           B=H+[-b_0w,b_0w]g.
\]
For each \(z\in B\), at least \(|H|(b_0w+1)>|B|/2\) choices of \(v\in B\) satisfy \(z-v\in B\). This follows by overlapping the two integer intervals and then filling the \(H\) fibres. At most \(0.04|B|\) choices are lost by insisting on \(u_1+v\in X_1\) and \(u_2+z-v\in X_2\). Thus
\begin{equation}
             u_1+u_2+B\ \subseteq\ X_1+X_2
                           =\Sigma_0(U_1\cup U_2).              \label{eq:exact1}
\end{equation}
Disjointness of the two selected sets is essential for this last equality. In particular \eqref{eq:exact1} gives entire consecutive fibres from positions independent of all the nonzero terms of \(A\).

Hold back one such term, and append all the others to \eqref{eq:exact1}. Each has a signed quotient coordinate \(z\) of magnitude at most \(2C_1w\). The interval of full fibres extends before reduction by exactly \(|z|\) at that step, since its length exceeds \(|z|\). Its length can never reach \(q\), since that would cover \(G\) while one held-back term remains to be cancelled. Writing \(h=|H|\), we obtain
\begin{equation}
                M(S)\ge h(1+\sum|z|).                          \label{eq:exact2}
\end{equation}
At most \(E_k=O(\sqrt k\log k)\) positions are omitted from the sum. No splitting of the group extension is assumed.

Write \(a=|S\cap H|\) and \(v\) for the number of nonzero heights in \eqref{eq:exact2}. We have \(a\le3\sqrt h\) by \eqref{eq:basic} and \(a<4k/5\) since \(H<G\). Negating terms with negative height makes the magnitudes positive. At most \(h\) terms can have any one positive magnitude, since an equality after negation would have come from a repeated term or an inverse pair. Thus \(v\ge k-a-E_k\) and
\begin{equation}
              M(S)\ge (v^2+hv)/2,\qquad M(S)\ge hv.            \label{eq:exact3}
\end{equation}
Indeed \(h\sum_{j=1}^v\lceil j/h\rceil=(v^2+hv+t(h-t))/2\) for \(t\equiv v\bmod h,\ 0\le t<h\). If \(h/k\to\infty\) then \(hv\ge h(k/5-o(k))\) contradicts the bound. If \(h/k\) is bounded above and away from zero, then \(k-v=o(h)\), also \(o(k)\), and \eqref{eq:exact3} gives \(2M(S)\ge k^2+(1-o(1))hk\), again impossible. Hence \(h=o(k)\) along any surviving sequence.

\section{The case of a trivial subgroup}\label{sec:trivial}

Here $h=1$ and $G$ is cyclic. Choose least signed representatives in units of \(g\), and replace each by its absolute value. The resulting \(k\) positive integers are distinct by freeness. Negating selected positions translates their subset support in \(G\) and exchanges chosen and unchosen status in precisely those positions. Thus the oriented set \(D\) has some residue with exactly one subset representation: the image of the original empty subset. This remains true even if \(M(S)=|G|\).

The \(k-O(\sqrt k\log k)\) positions counted in \eqref{eq:exact2} have distinct absolute values with sum at most \(T_k+1\). They contain all but \(o(k)\) values from \([1,k]\). In detail, let \(v\) be their count and \(u\) the number among them above \(k\). Their sum is at least \(T_{v-u}+uk+T_u=T_v+u(k-v)+u^2\), whereas \(T_k-T_v=o(k^2)\). Hence \(u=o(k)\). Put \(B_0=D\cap[1,k]\), \(m=k-|B_0|=o(k)\), and \(W_0=\sum B_0=(1/2-o(1))k^2\).

We claim that all integers in \([d,W_0-d]\), with \(d=5(m+1)\), have at least two distinct representations from \(B_0\). Put \(t=30(m+1)<k\). Every integer in \([d,2t-d]\) has at least \(m+2\) representations as an unordered pair of distinct members of \([1,t]\) for large \(k\), including \(m=0\). Indeed below the midpoint the pair count is \(\lfloor(x-1)/2\rfloor\), and reflection in \(t+1\) handles larger sums. For any fixed sum its pairs are disjoint, so at least two survive removal of the missing integers. Append the elements \(y_1<y_2<\cdots\) of \(B_0\) above \(t\). The initial twice-represented interval has length \(2t-2d+1\), and \(y_j\le t+m+j\). Successive translates overlap or touch, since its length initially exceeds \(t+m+1\) and then increases by at least one at each step. This produces a twice-represented interval beginning at \(d\) and ending at least at \(W_0-T_t\). Complementation reflects this interval about \(W_0/2\); the two intervals overlap for large \(k\) since \(T_t=o(k^2)\). This proves the claim.

Reduce modulo \(n=|G|\). The length of the twice-represented interval must be less than \(n\), as otherwise the uniquely represented residue would be impossible even after adding all remaining terms (one can omit them). Add those remaining absolute values one at a time. If the next weight \(v\) exceeds the current interval length, the interval and its \(v\)-translate are disjoint modulo \(n\), since \(v\le n/2\). Their combined size is at least \(2(W_0-2d+1)>T_k+1\). Otherwise the interval extends continuously in the integers and remains twice represented. It can never reach length \(n\). With \(W=\sum D\), this proves
\[
                 M(S)\ge W-2d+1,\qquad W\le T_k+2d.
\]
If \(D\) missed a positive integer \(j\le\lfloor k/2\rfloor\), the cheapest possible \(k\) distinct positive weights avoiding \(j\) would have sum \(T_k+(k+1-j)\). Since \(2d=o(k)\), this is impossible. Thus \(D\) contains \([1,\lfloor k/2\rfloor]\), and its largest element is at most \(W-T_{k-1}\le k+2d\).

The integer subset sums of \([1,L]\), for \(L\ge4\), fill \([0,T_L]\), with at least two representations at every height from 3 through \(T_L-3\). For \(L=4\), heights 3 through 7 each have two representations by direct addition. Appending \(L+1\) extends both intervals, since \(L+1\le T_L-5\) for \(L\ge4\). Start now with \(L=\lfloor k/2\rfloor\). The remaining weights of \(D\) are all at most \(k+2d<T_L-5\), and the same extension argument shows that the support in the integers is exactly \([0,W]\), with all heights in \([3,W-3]\) twice represented. If \(W\ge n+2\), residues \(0,1,2\) have witnesses at both \(j\) and \(j+n\), and residues \(3,\ldots,n-1\) lie in the twice-represented interval. This contradicts the uniquely represented residue. We have proved \(W\le n+1\) and \(\N(S)=\min(W,n-1)\).

For every \(n\) in the case \(h=1\), the inequalities \(T_k\le W\le n+1\) give \(n\ge T_k-1\) and \(\N(S)\ge T_k-2\). If \(n\ge T_k+2\), the hypothesis \(\N(S)\le T_k\) and the identity \(\N(S)=\min(W,n-1)\) force \(W=T_k\), hence \(D=[1,k]\). The empty subset and the singletons of weights 1 and 2 uniquely represent the first three heights. Complementation gives uniqueness at the last three. Thus the only uniquely represented integer heights are \(0,1,2,T_k-2,T_k-1,T_k\), and \(W<n\) prevents modular identifications. The original negative positions are precisely the subset representing one of those six heights. Changing the generator's sign complements that subset. Thus, up to generator, the set is the initial segment \([1,k]\), that segment with 1 negated, or that segment with 2 negated. Each is free by toggling the corresponding unique representation.

\section{The case of a nontrivial subgroup}\label{sec:nontrivial}

Assume now $h\ge2$.

Orient every nonzero quotient height to its least positive absolute value by negating the corresponding terms in \(G\). For a self-opposite nonzero quotient height retain the original terms. Leave the kernel positions alone. Each positive height has at most \(h\) positions, since negation cannot create a repeated term from a zero-sum-free ordinary set. The oriented support again has a uniquely represented residue, by toggling positions. Let \(t=\lfloor k/h\rfloor\to\infty\), and let \(W\) denote the sum of the oriented positive integer heights of all positions, omitting height zero.

The \(v\) positive heights from \eqref{eq:exact2} have \(v=k-o(k)\), since \(a\le3\sqrt h\). Their total is at most \((T_k+1)/h\). If \(u\) of them exceed \(t\), their weight is at least
\[
 \frac{(v-u)^2}{2h}+tu+\frac{u^2}{2h}
 \ \ge\ \frac{v^2}{2h}-u+\frac{u^2}{h}.
\]
The two quadratic terms bound the cost of filling the lowest available slots in the respective ranges, with at most \(h\) entries per height. The displayed inequality then uses \(t-v/h\ge-1\), since \(v\le k\) and \(t=\lfloor k/h\rfloor\). Hence \(u^2\le T_k+1-v^2/2+hu=o(k^2)\). Among the \(ht\) available positions at heights \(1,\ldots,t\), only \(o(k)\) are missing. These are actual, distinct quotient heights: \eqref{eq:exact3} gives \(hq\ge v^2/2\), so \(q/t\to\infty\).

All but \(o(t)\) of those heights contain more than \(3h/4\) positions. Choose two such distinct heights \(j_1,j_2=o(t)\), for instance the two earliest. By choosing exactly one term at each, all elements of the \(H\)-coset at height \(j_1+j_2\) have at least two representations. In fact each target in the coset has at least \(|B_1|+|B_2|-h>h/2\) ordered choices from the two respective cells. The cells are distinct and \(h\ge2\).

Take one representative from each other occupied height up to \(t\). Its set of integer heights omits only \(o(t)\) integers of \([1,t]\). By the elementary integer interval argument of Section~\ref{sec:trivial}, its subset heights contain \([d,V-d]\), where \(d=o(t)\) and \(V=(1/2-o(1))t^2\) is their sum. Combining these representatives with the preceding two cells makes every element of every fibre in the translated interval at least twice represented. Add all other positions of height at most \(t\), except those remaining in the two designated cells. Each new weight is at most \(t\), so the interval extends without gaps. The same count works with integer heights before reduction by \(q\): in the fibre at integer height \(z\) we keep track of the actual group sums above \(z\bar g\). Translation by a position and complementation take full fibres to full fibres and preserve multiplicity even in a nonsplit extension.

Write \(W_0\) for the sum of all positive heights at most \(t\), including the two designated cells. There are \(k-o(k)\) positions in this range, so \(hW_0\ge(1/2-o(1))k^2\) by the capacity estimate. The interval just obtained begins at \(d+j_1+j_2=o(t)\) and ends at least at \(W_0-2ht-d\). Since \(ht=O(k)=o(W_0)\), it overlaps its reflection under complementation in all the positions of heights \(1,\ldots,t\). Hence for some \(e_k=o(t)\) every element in each fibre of the integer interval \([e_k,W_0-e_k]\) has at least two representations.

Its length must be less than \(q\) by the unique representation in the oriented support. Append the rest of the positive heights. If the next height exceeds the interval length, the two arcs in the cyclic quotient are disjoint, since the height is at most \(q/2\); their full fibres give at least \(2h(W_0-2e_k+1)>T_k+1\) sums. Otherwise the interval extends by that height. It never covers the entire quotient, again by uniqueness. Zero-height positions may be omitted throughout. We obtain
\begin{equation}
            M(S)\ge h(W+1-2e_k)=h(W+1)-o(k).                \label{eq:exact4}
\end{equation}
For \(v'=k-a\) and \(0\le b<h,\ b\equiv v'\pmod h\), the least possible \(W\) is \(W_h(v')=\sum_{j=1}^{v'}\lceil j/h\rceil\). In particular
\begin{equation}
 h(1+W_h(k-a))=\frac{k^2}{2}
       +\left(\frac h2-a\right)k
       +h+\frac{a^2-ha+b(h-b)}2 .                          \label{eq:exact5}
\end{equation}

If \(h\to\infty\), use \(a\le3\sqrt h\) and \(h=o(k)\) in \eqref{eq:exact4}--\eqref{eq:exact5}. The positive term \((h/2-o(h))k\) exceeds the error and the possible \(O(h^{3/2})\) negative constant term. This contradicts \(M(S)\le T_k+1\). If \(h\ge43\) is fixed, then \(a\le3\sqrt h<(h-1)/2\). Equation \eqref{eq:exact5} again exceeds the triangular bound after \eqref{eq:exact4}. Thus, under the triangular bound, \(H=0\) whenever all prime divisors of \(G\) are at least 43 and \(k\) is sufficiently large. This conclusion is uniform in the rank and group order.

\section{Exact lower bound and equality}\label{sec:equality}

Now suppose \(\N(S)\le F(k)\), so \(M(S)\le\lfloor k^2/2\rfloor+2\). The case \(h=1\) was excluded at this smaller bound by Section~\ref{sec:trivial}. By the preceding arguments any remaining sequence of sets within this bound may, after passage to a subsequence, be taken to have fixed \(h\ge2\) and fixed \(a\). For a zero-sum-free set of size $a\ge4$ we have $\N(A)\ge2a$. Indeed, the known values of $f$ at sizes $4,5,6,7$ are $8,13,19,24$ \cite{GLPS,YZ}, and every $a\ge8$ is a sum of block sizes from $4$ through $7$. The assertion follows by \eqref{eq:superadditive}. At smaller sizes $f(1)=1$, $f(2)=3$, and $f(3)=5$ \cite[Lemma~2.2]{GLPS}. Applied inside \(H\), these give \(a\le h/2\). Equality is possible only at \((h,a)=(2,1),(4,2),(6,3)\), and in every one of these three cases the subset sums of the \(a\) terms fill \(H\), since their number is at least \(h\). If \(a<h/2\), \eqref{eq:exact4}--\eqref{eq:exact5} already give \(M(S)\ge k^2/2+k/2-o(k)\), contradicting this bound.

In the three equality cases, let \(v'=k-a\). The quotient sequence of the exterior terms, before orientation, is zero-sum-free: any nonempty selection with sum in \(H\) could be cancelled with the independent kernel positions. Furthermore, the full inclusive count in \(G\) is exactly \(h\) times the quotient count. The oriented quotient sequence therefore has a uniquely represented residue. From \eqref{eq:exact4}--\eqref{eq:exact5} and this bound we get
\[
                         W-W_h(v')=o(k).
\]
Every positive height \(j\le L=\lfloor k/(2h)\rfloor\) must now have all \(h\) positions: missing a slot there adds at least \(\lfloor v'/h\rfloor-j\) to the minimum weight, which is linear in \(k\). This follows by comparing with the cheapest filling of the \(v'\) slots, or by forbidding that one slot in the ordered list in which every positive integer occurs \(h\) times. Also the maximum height is at most \(W-W_h(v'-1)=\lceil v'/h\rceil+o(k)\).

Two copies of 1 give all integer heights in \([0,2]\) and at least two representations at the interior height. If \([0,A]\) is filled with at least two representations throughout its interior, adding a positive weight \(x\le A\) gives the same property on \([0,A+x]\). The overlap supplies two witnesses wherever the two endpoints meet. Append first the rest of the full levels \(1,\ldots,L\) in increasing order, then all the other weights. The running total after the full levels is quadratic in \(k\) for fixed \(h\), so every later weight fits. This yields all heights \([0,W]\), with at least two representations at every interior height. If \(W\ge q\), then every residue has at least two representations: when \(W=q\) the two endpoints coincide, and when \(W>q\) the same conclusion follows from the interior and the second witness for residue zero. Therefore \(W<q\), and \(M(S)=h(1+W)\) with no loss at the endpoints.

At \(a=h/2\), \eqref{eq:exact5} reduces to
\[
               M(S)\ge k^2/2+h-h^2/8+b(h-b)/2.
\]
For \(h=2,4,6\) the constant \(h-h^2/8\) is respectively \(3/2,2,3/2\). Integrality gives \(M(S)\ge\lfloor k^2/2\rfloor+2\) for both parities of \(k\), the required lower bound in every case.

The sequence argument establishes an absolute threshold: if assertions in the proof failed at arbitrarily large sizes, choose such a sequence, pass to a subsequence in each of the stated size ranges for \(h\), and for bounded \(h\) pass to a fixed \(h,a\). Each alternative has been treated. For the upper bound take an even cyclic group of order \(2(\lfloor k^2/4\rfloor+1)\), its involution, and both lifts modulo that involution at each height \(1,\ldots,\lfloor(k-1)/2\rfloor\); for even \(k\) add one lift at height \(k/2\). All positive heights total \(\lfloor k^2/4\rfloor\), strictly less than the quotient order. All intermediate heights occur and the involution gives both lifts. This proves Theorem~\ref{th:exact}.

At equality, \(M(S)=h(1+W)\) as proved above. The lower bound \(h(1+W_h(k-a))\) is itself at least \(F(k)+1\), by \eqref{eq:exact5} and integrality. Hence \(W=W_h(k-a)\). All the smallest positive quotient heights must be filled to capacity \(h\), with a possible partial last level. The oriented quotient has only two uniquely represented heights, 0 and \(W\): all internal heights were proved twice represented, and \(W<q\) rules out wraparound. Toggling the original negative positions must give a uniquely represented height because the original exterior sequence in the quotient was zero-sum-free. The negative positions are therefore empty or all the exterior positions. Changing the quotient generator's sign if necessary gives precisely the positive levels in Theorem~\ref{th:structure}. In \eqref{eq:exact5}, equality to \(F(k)+1\) permits every \(k\) for \(h=2\); for \(h=4\) it requires even \(k\) and \(b=0\); for \(h=6\) it requires odd \(k\) and \(b=0\). Conversely, if the kernel set is zero-sum-free, the displayed configurations are zero-sum-free because the total exterior height is less than \(q\). The kernel supplies every lift, and the successive quotient levels realize every height through \(W\). All listed kernel orders occur. In cyclic kernels of orders \(2,4,6\) one may take the residues \(\{1\},\{1,2\},\{1,3,4\}\), respectively, and form a direct product with a cyclic quotient of order greater than \(W\). This proves Theorem~\ref{th:structure}.

\section{Cyclic groups of prime-power order}\label{sec:primepower}

Independently of the rank or order, Sections~\ref{sec:progression}--\ref{sec:nontrivial} show the following for all sufficiently large \(k\). If every prime divisor of the generated group is at least 43 and \(\N(S)\le T_k\), then that group is cyclic of some order \(n\ge T_k-1\). In the normalization of Section~\ref{sec:trivial} the positive absolute weights have total \(T_k\le W\le n+1\) and \(\N(S)=\min(W,n-1)\). For \(n\ge T_k+2\) the three signed initial segments in Section~\ref{sec:trivial} are the full list up to generator. In particular the lower bound \(\N(S)\ge T_k-2\) holds with no triangular bound assumption under this prime-divisor restriction.

Take \(p\ge43\). In \(C_{p^\nu}\), whenever \(T_k<p^\nu\) the first \(k\) multiples of a generator give \(\N=T_k\). Whenever \(T_k\ge p^\nu\) every feasible set would have \(\N\le T_k\). Thus it suffices in both cases to examine a free set within the triangular bound. For large \(k\) the preceding inverse conclusion supplies a subgroup order \(n=p^j\ge T_k-1\). Neither \(T_k\) nor \(T_k-1\) can be a power of \(p>3\) at large \(k\): factor them as \(k(k+1)/2\) and \((k-1)(k+2)/2\). After dividing the even factor by two both remaining factors exceed one, and their gcd divides 1 in the first factorization and 3 in the second. Therefore \(n>T_k\), so both \(W\) and \(n-1\) are at least \(T_k\). The identity \(\N(S)=\min(W,n-1)\) now gives the claim, proving necessity of \(T_k<p^\nu\) and the lower bound \(\N(S)\ge T_k\) throughout the feasible range. All thresholds in this argument are independent of \(\nu\).

For the remaining bounded sizes we use a finite-pattern argument. For each fixed \(k\) there are finitely many subset-sum equality patterns. Any such pattern has a universal finitely generated abelian group \(U=\mathbb Z^k/L\), where \(L\) is generated by the differences of incidence vectors of subsets identified by the pattern. Let \(p\) exceed all prime divisors of the torsion orders of these finitely many groups (and exceed \(T_k\)). For any realization of a free ordinary pattern in \(C_{p^\nu}\), the natural map from \(U\) to the realization kills all torsion, so factors through the free quotient. The images of the \(k\) positions in the free quotient are distinct, nonzero and have no inverse pairs. Two subsets have the same sum in the free quotient if and only if they have the same sum in the realization. Indeed, each equality in the realization defines a relation in \(L\), and equality in the free quotient is preserved by the map to the realization. Choose an integer linear functional that separates its finitely many different subset sums, by avoiding finitely many proper rational hyperplanes. The resulting \(k\) integers have distinct nonzero absolute values. Changing signs preserves the support size. For positive integers \(d_1<\cdots<d_k\), adding \(d_i\) to the first \(i-1\) positions adds at least \(i\) new sums: writing \(V=\sum_{j<i}d_j\), the numbers \(V+d_i\) and \(V+d_i-d_j\) for \(j<i\) are distinct and all exceed \(V\). Hence \(\N\ge T_k\). Choose \(p_0\) to handle these finitely many \(k\) at once. This proves Theorem~\ref{th:pgroup}.

\appendix
\section{The explicit threshold for the minimum}\label{sec:explicit}

We summarize the calculation that makes the threshold in Theorem~\ref{th:exact} explicit. Write $B_k=\lfloor k^2/2\rfloor+1$. The construction in \eqref{eq:classical-upper} gives $f(k)\le B_k$, so a counterexample to equality must satisfy $M(S)\le B_k$, where $M$ includes the empty sum. The calculation first establishes the lower bound for $k\ge195$ and then treats the ten sizes $185\le k\le194$. It uses unconditional lower bounds at smaller sizes throughout.

Fix a root size $k$, write $B=B_k$, and suppose $M(S)\le B$.
Write $Y_j$ for an unconditional lower bound for the inclusive support
of a free $j$-element set. After deletions with cumulative loss at least
$L$, the remaining set $U$ of size $m$ satisfies $M(U)\le B-L$.
Thus $Y_m+L>B$ excludes the branch. Generation must also be checked:
if the last deleted element lies outside $\langle U\rangle$, its two
translates of $\Sig(U)$ are disjoint. The support immediately before
that deletion then had size at least $2Y_m$.

Here is the estimate that supplies deletion losses. Put
$A=\{0\}\cup U\cup(-U)$, let $Z$ be the smaller of $\Sig(U)$ and its
complement in $\langle U\rangle$, and write $z=|Z|$. Let $D$ be the
largest loss on deleting one element of $U$. For $u\in U$, writing
$X=\Sig(U\setminus\{u\})$ gives $\Sig(U)=X\cup(X+u)$. Its boundary
in direction $u$ has size at most $|(X+u)\setminus X|\le D$.
The same bound holds for $Z$ by \eqref{eq:boundary-properties}.
Consequently $\lambda_Z(x)\le rD$ for $x\in rA$. If $z>0$, choose
the $2z-1$ nonzero group elements of smallest word length with
respect to $A$. Their total boundary is at least
$(2z-1)z-z(z-1)=z^2$, by counting ordered pairs of distinct elements
of $Z$. Summing their word lengths by layers yields
\[
 z^2\le D\sum_{r\ge0}(2z-|rA|)_+,\qquad 0A=\{0\},
 \qquad x_+=\max(x,0).
\]
Lower estimates on the sizes of the balls therefore give a lower
bound on $D$, rounded up before the next deletion.

When a cyclic quotient is obtained, let $h$ be the kernel order and
$q$ the quotient order. Orient the exterior positions to positive
absolute heights as in Section~\ref{sec:nontrivial}. If there are $n$
such positions, their total height $W$ is at least
$W_h(n)=\sum_{j=1}^n\lceil j/h\rceil$, since each height holds at
most $h$ positions. Direct counts in distinct quotient cosets may be
added. Alternatively, suppose every point of each fibre at integer
heights $l,\ldots,W-l$ has two representations. Its width
$w=W-2l+1$ is less than $q$, since the oriented set has a uniquely
represented sum. On restoring an exterior position of absolute height
$v\le q/2$, the interval extends if $v\le w$. If $v>w$, the two
intervals are disjoint modulo $q$ and give at least $2hw$ sums.
Thus if the budget at the enlarged size is less than $2hw$, only
extension is possible. A restored kernel position preserves the interval.
This comparison is applied at the budget for each successive size.
In particular, if the ultimate interval alone contains more than $B$
points, the original set cannot be a counterexample.

There is also a direct finite formulation of what must be excluded at
each fixed size. For a partition $\mathcal P$ of the subsets of $[k]$,
let $L_{\mathcal P}\le\mathbb Z^k$ be generated by
$\mathbf1_I-\mathbf1_J$ for $I,J$ in the same block. A counterexample
at size $k$ exists if and only if there is a partition into at most
$B$ blocks with the following properties: the empty subset is alone
in its block, the $k$ singletons are in distinct blocks, and
$\mathbf1_I-\mathbf1_J\notin L_{\mathcal P}$ for every pair from
different blocks. Necessity follows by mapping the standard basis
to the elements of $S$. Conversely, pass to
$\mathbb Z^k/L_{\mathcal P}$, in which the prescribed differences
between blocks are nonzero. They remain simultaneously nonzero in
some finite quotient: retain the torsion summand and reduce the free
coordinates modulo an integer larger than every absolute coordinate
of these finitely many differences. The basis images in that quotient
realize exactly the prescribed partition. This explains why the
fixed-size question does not require any a priori bound on the
ambient group order. The deletion and quotient estimates above reduce
the computation substantially.

The root checks separate the cases $|G|-M(S)\le7$ and $|G|-M(S)\ge8$, where $G=\langle S\rangle$. At $k=185$ the upper limit on $M(S)$ for a counterexample is $B_{185}=17113$. The finite checks exclude this limit in both ranges of $|G|-M(S)$. Hence $M(S)\ge17114$, and the matching construction gives $f(185)=17113$.

\end{document}